\documentclass[a4paper,12pt]{article}
\usepackage{graphicx}
\usepackage[cp1251]{inputenc}
\usepackage[ukrainian, russian, english]{babel}

\usepackage{amsmath, amsthm, amsfonts, amssymb}
\usepackage{bbm}
\usepackage{tikz} 
\usepackage{tikz-cd}
\newtheorem{theorem}{Theorem}
\newtheorem{lemma}{Lemma}
\newtheorem{example}{Example}
\newtheorem{remark}{Remark}
\newtheorem{definition}{Definition}

\title{Feynman diagrams and their limits for Bernoulli random variables}
\author{Anastasiia Hrabovets}

\date{April 2023}

\begin{document}

\maketitle

\section{Introduction}
    
    According to the central limit theorem, the Gaussian distribution is a limit of the distributions of the normalized sums of independent Bernoulli random variables. 
    The survey \cite{Privault_1} presented the construction of basic operators of stochastic analysis, such as divergence and gradient, for the Bernoulli process, studied the chaos representation property, presented the application to functional inequalities, such as deviation inequalities and logarithmic Sobolev inequalities,  in discrete settings. The study of a concentration property of the distributions of the weighted sums with Bernoullian coefficients was used to derive an almost surely version of the central limit theorem~\cite{Bobkov}.
    
    In this article, we will construct an approximation of Gaussian white noise based on the sequence of Bernoulli random variables and define Wick's products and the stochastic exponent for the Bernoulli case. Here we will propose a method to calculate the moments of Wick's products for Bernoulli variables using diagrams, that converge to Feynman diagrams in the Gaussian case.
    
    The space of square-integrable functions measurable with respect to Gaussian white noise is an important tool for describing quantum systems because the decomposition of this space into the direct sum of subspaces of orthogonal polynomials is isomorphic to the Fock space \cite{Simon}. Naturally, the question of such a decomposition for the space of functions from a sequence of Bernoulli random variables arises. We will prove that orthogonal polynomials for Bernoulli noise converge to Hermite polynomials, which form an orthogonal system in the Gaussian case.
	
	\section{Bernoulli noise}
	
	In this section, we consider an analog of Gaussian white noise based on Bernoulli random variables. 
	
	Let $\left\{\varepsilon_n, n\geq 1\right\} $ be a sequence of independent random variables with Bernoulli distribution:
	$$ \varepsilon = 
	\begin{cases}
		1, \;\;\frac{1}{2}\\
		-1, \frac{1}{2}
	\end{cases} $$
	
	Define $$\forall f \in C\left([0,1]\right):~ \varphi\left(f\right)  := \sum\limits_{k=0}^n f\left(\frac{k}{n}\right)\frac{\varepsilon_k}{\sqrt{n}}$$
	
	\begin{definition}\label{def_bernoulli_noise}
		We will call a set $\left\{\varphi(f): f \in C\left([0,1]\right)\right\}$ of random variables by the Bernoulli noise in $C\left([0,1]\right)$.
	\end{definition}
	
	\par Let us consider the properties of the Bernoulli noise.
	
	\begin{lemma}\label{lem_moments_powers}
		For any $ f_1, \ldots, f_j \in L_2([0,1])$ and $ q_1, \ldots q_j  \in \mathbb{N}$ following statement is true:
		\begin{gather}\label{1}
			E \left[\varphi^{q_1}(f_1) \cdot \ldots \cdot \varphi^{q_j}(f_j) \right] = \nonumber \\
			= \sum\limits_\Pi \prod\limits_{D \in \Pi}  \sum\limits_{k=1}^n \sum_{p = 1}^\infty (-1)^{p+1} \frac{|B_{2p}|2^{2p}\left(2^{2p} -1\right)}{2p} \frac{1}{(2p - |D|)!} \prod\limits_{d \in D} \left(f_d\left(\frac{k}{n}\right)\frac{1}{\sqrt{n}}\right) \cdot \nonumber \\
			\cdot \left(\sum\limits_{m=1}^j \lambda_m f_m \left(\frac{k}{n}\right)\frac{1}{\sqrt{n}}\right)^{2p - |D|} \mathbbm{1}_{(2p-|D| \geqslant 0)}, ~~\overline{\lambda} = 0,
		\end{gather}  where $\Pi$ is a set of all possible partitions of a set $\left\{1_1, 1_2, \ldots 1_{q_1}, 2_1, \ldots, 2_{q_2}, \ldots, j_{q_j}\right\}$, $D$ is one of possible partitions, $B_i$ are Bernoulli numbers.
	\end{lemma}
	\begin{proof}
		Let $ \overrightarrow{\lambda} \in \mathbb{R}^j.$ Then
		$$E e^{\left(\overrightarrow{\lambda}, \left(\varphi(f_1) \cdot \ldots \cdot \varphi(f_j)\right)\right)} = E \exp\left(\sum_{m = 1}^j \lambda_m \varphi(f_m)\right) = $$
		$$ = E \exp \left(\sum\limits_{k=1}^n \left(\sum\limits_{m=1}^j \lambda_m f_m\right) \left(\frac{k}{n} \right) \frac{\varepsilon_k}{\sqrt{n}} \right)= \prod\limits_{k=1}^n E \exp\left(\sum\limits_{m=1}^j \lambda_m f_m \left(\frac{k}{n} \right) \frac{\varepsilon_k}{\sqrt{n}}\right) = $$ 
		$$ = \prod\limits_{k=1}^n \frac{1}{2} \left(\exp\left(\sum\limits_{m=1}^j \lambda_m f_m \left(\frac{k}{n} \right) \frac{1}{\sqrt{n}}\right) + \exp\left(-\sum\limits_{m=1}^j \lambda_m f_m \left(\frac{k}{n} \right) \frac{1}{\sqrt{n}}\right)\right) =$$
		
		$$ = \prod\limits_{k=1}^n \cos \left(\sum\limits_{m=1}^j i \lambda_m f_m \left(\frac{k}{n} \right) \frac{1}{\sqrt{n}}\right) = \exp\left(\sum\limits_{k=1}^n \ln\left(\cos\left(\sum\limits_{m=1}^j i \lambda_m f_m \left(\frac{k}{n}\right)\frac{1}{\sqrt{n}}\right)\right)\right) = $$
		$$ = \exp\left(-\sum\limits_{k=1}^n \sum_{p = 1}^\infty \frac{|B_{2p}|2^{2p}\left(2^{2p} -1\right)}{2p(2p)!} \left(\sum\limits_{m=1}^j i \lambda_m f_m \left(\frac{k}{n}\right)\frac{1}{\sqrt{n}}\right)^{2p}\right) = $$
		$$ = \exp\left(\sum\limits_{k=1}^n \sum_{p = 1}^\infty (-1)^{p+1} \frac{|B_{2p}|2^{2p}\left(2^{2p} -1\right)}{2p(2p)!} \left(\sum\limits_{m=1}^j \lambda_m f_m \left(\frac{k}{n}\right)\frac{1}{\sqrt{n}}\right)^{2p}\right)$$
		\par Now note that
		$$E \left[\varphi(f_1)^{q_1} \cdot \ldots \cdot \varphi(f_j)^{q_j} \right] = \frac{\partial^q}{\partial \lambda_1^{q_1} \cdot \ldots \cdot \lambda_j^{q_j}} E e^{\left(\overline{\lambda}, \left(\varphi(f_1) \cdot \ldots \cdot \varphi(f_j)\right)\right)},$$
		where $q = \sum q_i$.
		
		\par Using Faa di Bruno's formula for the derivatives of the composition of two real-valued functions \cite{Bruno} one can get:
	=
			$$E \left[\varphi^{q_1}(f_1) \cdot \ldots \cdot \varphi^{q_j}(f_j) \right] = $$
			$$= \sum\limits_\Pi \prod\limits_{D \in \Pi}  \sum\limits_{k=1}^n \sum_{p = 1}^\infty (-1)^{p+1} \frac{|B_{2p}|2^{2p}\left(2^{2p} -1\right)}{2p} \frac{1}{(2p - |D|)!} \prod\limits_{d \in D} \left(f_d\left(\frac{k}{n}\right)\frac{1}{\sqrt{n}}\right) \cdot $$
			$$\cdot \left(\sum\limits_{m=1}^j \lambda_m f_m \left(\frac{k}{n}\right)\frac{1}{\sqrt{n}}\right)^{2p - |D|} \mathbbm{1}_{(2p-|D| \geqslant 0)}, ~~\overline{\lambda} = 0,$$
		
		where $\Pi$ is a set of all possible partitions of a set $\left\{1_1, 1_2, \ldots 1_{q_1}, 2_1, \ldots, 2_{q_2}, \ldots, j_{q_j}\right\}.$
	\end{proof}

	The obtained random variables are analogous to Gaussian white noise. The previous lemma allows us to calculate the moments of the products of these random variables. In the case of Gaussian random variables, one can define the stochastic exponent and Wick's power for random variables. Let's try to do the same for the Bernoulli noise.

	\begin{definition}\label{def_Wick_power}
		Let $\eta$ be a random variable with finite moments. Then $n$th Wick's power $:\eta^n: , n=0, 1, \dots$ is a $n$th-order  polynomial $P_n(\eta)$, which is determined by the relations:
		$$P_0(x) = 1,$$
		$$\frac{\partial}{\partial x} P_n(x) = n P_{n-1}(x), n = 1, 2, \dots,$$
		$$E \left(P_n(\eta)\right) = 0, n = 1, 2, \dots.$$
	\end{definition}

	\begin{definition}\label{def_stochastic_exponent}
		The exponential generating function of a sequence $\left\{:\eta^n:\right\} $ is a power series of the form
		$$:e^{\alpha\eta}:~ = \sum\limits_{n=0}^{\infty}\frac{\alpha^n:\eta^n:}{n!}$$
		It is called the stochastic exponent (or Wick's exponent) for the random variable $\eta$.
	\end{definition}

	\begin{example}
		To show the existence of the stochastic exponent for a Bernoulli random variable, we need to prove that the series given in definition \ref{def_stochastic_exponent} converges almost surely.
		
		Denote $f_n(x) = \frac{\alpha^n P_n(x)}{n!}$ and determine the properties of $f_n$ using the definition of Wick's power:
		
		$$f_0(x) = 1,$$
		$$f_{n+1}^{'}(x) = \frac{d}{dx}\left(\frac{\alpha^{n+1} P_{n+1}(x)}{(n+1)!}\right) = \alpha\left(\frac{\alpha^n P_n(x)}{n!}\right) = \alpha f_n(x).$$
		
		Let's assume $$f_{2n}(x) = \sum\limits_{k=0}^n b_{2k}^{(2n)} x^{2k}.$$
		
		Then
		
		$$f_{2n+1}(x) = \alpha\int f_{2n}(x) dx = \sum\limits_{k=0}^n \alpha b_{2k}^{(2n)} \frac{x^{2k+1}}{2k+1} + C_1.$$ 
		
		Since the mathematical expectation of odd powers of Bernoulli random variables is equal to zero, then 
		$$C_1 = 0,$$
		$$f_{2n+1}(x) =  \sum\limits_{k=0}^n \alpha b_{2k}^{(2n)} \frac{x^{2k+1}}{2k+1}.$$
		
		Similarly
		
		$$f_{2n+2}(x) = \alpha\int f_{2n+1}(x) dx = \sum\limits_{k=0}^n \alpha^2 b_{2k}^{(2n)} \frac{x^{2k+2}}{(2k+1)(2k+2)} + C_2. $$
		
		Since the mathematical expectation of even powers of Bernoulli random variables is equal to one, then 
		$$C_2 = - \sum\limits_{k=0}^n  \frac{\alpha^2 b_{2k}^{(2n)}}{(2k+1)(2k+2)}.$$
		
		From these relations, one can get that \\
		$ \forall n, m \in \mathbb{N}:$ $$|b_m^{2n}| \leqslant \sum\limits_{k=0}^{n-1}  \frac{|\alpha^2 b_{2k}^{(2n-2)}|}{(2k+1)(2k+2)} \leqslant  \sum\limits_{k=0}^{n-1}  \frac{|\alpha^2|}{(2k+1)(2k+2)}\sum\limits_{k=0}^{n-2}  \frac{|\alpha^2 b_{2k}^{(2n-4)}|}{(2k+1)(2k+2)} \leqslant$$ $$\leqslant \ldots \leqslant  \prod\limits_{k=0}^{n-1} |\alpha^2|\left(\frac{1}{1 \cdot 2}  + \ldots + \frac{1}{(2k+1)(2k+2)}\right).$$
		
		Let's check whether the series $\sum\limits_{n=0}^\infty |f_{2n}(\varepsilon)|$ converges
		$$\sum\limits_{n=0}^\infty |f_{2n}(\varepsilon)| = 1 +  \sum\limits_{n=1}^\infty |f_{2n}(\varepsilon)| \leqslant 1 + \sum\limits_{n=1}^\infty \sum\limits_{k=0}^n |b_{2k}^{(2n)}| \leqslant$$ $$ \leqslant 1 + \sum\limits_{n=1}^\infty n \prod\limits_{k=0}^{n-1}|\alpha^2| \left(\frac{1}{1 \cdot 2}  + \ldots + \frac{1}{(2k+1)(2k+2)}\right).$$
		
		Let's consider the sum:
		$$\sum\limits_{k=0}^{n}  \frac{1}{(2k+1)(2k+2)} \leqslant \sum\limits_{k=0}^{\infty}  \frac{1}{(2k+1)(2k+2)} = \frac{1}{2} + \sum\limits_{k=1}^{\infty} \frac{1}{4k^2 + 6k +2} = $$ $$ =  \frac{1}{2} + \frac{1}{4}\sum\limits_{k=1}^{\infty} \frac{1}{k^2 + \frac{3}{2}k + \frac{1}{2}} \leqslant \frac{1}{2} + \frac{1}{4}\sum\limits_{k=1}^{\infty} \frac{1}{k^2} = \frac{1}{2} + \frac{\pi^2}{24} < 0.92.$$
		
		In this case, if $|\alpha| \leqslant 1$, then:
		
		$$\sum\limits_{n=1}^\infty n \prod\limits_{k=0}^{n-1}|\alpha^2| \left(\frac{1}{1 \cdot 2}  + \ldots + \frac{1}{(2k+1)(2k+2)}\right) \leqslant \sum\limits_{n=1}^\infty n (0.92)^{n-1} \text{ ~~converges.}$$
		
		Hence, $\sum\limits_{n=0}^\infty |f_{2n}(\varepsilon)|$ converges.
		
		In a similar way one can show that $\sum\limits_{n=0}^\infty |f_{2n+1}(\varepsilon)|$ also converges.
		
		So, one can conclude that $:e^{\alpha\varepsilon}:$ converges a.s. if $|\alpha| \leqslant 1.$

	\end{example}

	\begin{lemma}\label{lemma_stochatic_exponent}
		If the stochastic exponent $:\exp\left(\alpha \eta\right):$ for random variable $\eta$ exists and converges for $\alpha \in \mathbb{R}$, it can be determined by the formula $$:\exp\left(\alpha \eta\right):~ = \frac{\exp\left(\alpha \eta\right)}{E \exp\left(\alpha \eta\right)}$$ 
	\end{lemma}
	\begin{proof}
		Denote $:\exp\left(\alpha \eta\right):~ = g\left(\eta\right)$. From the relations given in the definitions \ref{def_Wick_power} and \ref{def_stochastic_exponent} and assumptions about the existence and convergence of the stochastic exponent, we can determine the following properties of the function $g(x)$:
		$$\frac{\partial}{\partial x} g(x) = \alpha g(x)$$
		$$E \left(g(\eta)\right) = 1$$
		Then 
		$$\int \frac{1}{g\left(x\right)} \partial g\left(x\right)= \int \alpha \partial x$$
		$$g\left(x\right) = e^{\alpha x + C}$$
		Hence
		$$:\exp\left(\alpha \eta\right): ~= g\left(\eta\right) = e^{\alpha \eta + C}$$
		Let's define a constant $C$ from the condition $$E :e^{\alpha \eta}:~ = 1$$
		$$E g\left(\eta\right) = e^C E e^{\alpha \eta} = 1,$$	
		$$ e^C = \frac{1}{E e^{\alpha \xi}}.$$
		Hence $$:\exp\left(\alpha \eta\right):~ = \frac{\exp\left(\alpha \eta\right)}{E \exp\left(\alpha \eta\right)}.$$
	\end{proof}
	
	The following lemma gives us the expression for Wick's powers for Bernoulli noise.
	
	\begin{lemma}\label{lemma_Wick_bernuolli_noise}
		\begin{gather}\label{2}
			:\varphi^{m}(f):~= \sum\limits_{\pi \in \Pi} \prod\limits_{D \in \pi} \Bigg[\varphi(f) \lambda^{1-|D|} \mathbbm{1}_{(1-|D| \geqslant 0)} + \nonumber \\
			+ (-1)^{\frac{|D|}{2}} b_p \frac{1}{(2-|D|)!}\sum\limits_{k=1}^n \left(f\left(\frac{k}{n}\right) \frac{1}{\sqrt{n}}\right)^{|D|} \left( \lambda f\left(\frac{k}{n}\right) \frac{1}{\sqrt{n}}\right)^{2p-|D|} \mathbbm{1}_{(2p-|D| \geqslant 0)} \Bigg]_{\lambda = 0},
		\end{gather}
	where  $\Pi$ is a set of all possible partitions of a set $\left\{1_1, 1_2, \ldots, 1_m\right\}$ and $b_p = \frac{|B_{2p}|2^{2p}\left(2^{2p} -1\right)}{2p}.$
	\end{lemma}
	\begin{proof}
		
		The series
		$$\exp(\lambda \varphi(f)) = \sum\limits_{m=0}^{\infty}\frac{\lambda^m \varphi^m(f)}{m!} $$ converges a.s.
		
		Note that
		$$|\varphi(f)|^m \leq c^m,$$
		where $$c = \sqrt{n} \max\limits_{x \in [0,1]}|f(x)|.$$

		Using this fact and applying Lebesgue's dominated convergence theorem one can get
		$$\mathbb{E} \exp(\lambda \varphi(f)) = \mathbb{E}\sum\limits_{m=0}^{\infty}\frac{\lambda^m \varphi^m(f)}{m!} = \sum\limits_{m=0}^{\infty}\frac{\lambda^m \mathbb{E}\varphi^m(f)}{m!}$$ converges.

		Therefore, the ratio of the exponent and its mathematical expectation can also be written as a convergent series:
		$$\frac{\exp(\lambda \varphi(f))}{\mathbb{E} \exp(\lambda \varphi(f))} = \frac{\sum\limits_{m=0}^{\infty}\frac{\lambda^m \varphi^m(f)}{m!}}{\sum\limits_{m=0}^{\infty}\frac{\lambda^m \mathbb{E}\varphi^m(f)}{m!}} = \sum\limits_{m=0}^{\infty}\frac{\lambda^m P_m(\varphi(f))}{m!}.$$
		
		Now let's show that the polynomials $P_m(\varphi(f))$  satisfy the properties from the definition~\ref{def_Wick_power} of Wick's powers.
		
		$$\sum\limits_{m=0}^{\infty}\frac{\lambda^m \varphi^m(f)}{m!} = \sum\limits_{m=0}^{\infty}\frac{\lambda^m P_m(\varphi(f))}{m!} \sum\limits_{l=0}^{\infty}\frac{\lambda^l \mathbb{E}\varphi^l(f)}{l!}.$$
		
		Consider the coefficients at different degrees of $\lambda:$
		
		$\lambda^0:$
		$$1=P_0(\varphi(f)).$$
		
		$\lambda^1:$
		$$\varphi(f) = P_1(\varphi(f))+\mathbb{E}\varphi(f);$$
		$$P_1(x) = x-\mathbb{E}\varphi(f);$$
		$$\frac{dP_1(x)}{dx}=\frac{d}{dx}\left(x-\mathbb{E}\varphi(f)\right) = 1 = P_0(x);$$
		$$\mathbb{E}P_1(\varphi(f)) = \mathbb{E}\left(\varphi(f) -\mathbb{E}\varphi(f)\right)=0.$$
		
		Suppose that for $\lambda^{m-1}$ the following holds:
		$$P_{m-1}(x)=x^{m-1} - \sum\limits_{i=1}^{m-1}\frac{P_{m-1-i}}{(m-1-i)!}\frac{\mathbb{E} \varphi^i(f)}{i!}(m-1)!;$$
		$$\frac{dP_{m-1}(x)}{dx}=(m-1)P_{m-2}(x);$$
		$$\mathbb{E} P_{m-1}(\varphi(f)) =0.$$
		
		Then for  $\lambda^m$ one can get:
		
		$$\frac{\varphi^m(f)}{m!} = \sum\limits_{i=0}^m\frac{P_{m-1}(\varphi(f))}{(m-i)!}\frac{\mathbb{E}\varphi^i(f)}{i!};$$
		$$P_{m}(x)=x^{m} - \sum\limits_{i=1}^{m}\frac{P_{m-i}}{(m-i)!}\frac{\mathbb{E} \varphi^i(f)}{i!}(m)!;$$
		$$\frac{dP_{m}(x)}{dx}=m\left(x^{m-1} - \sum\limits_{i=1}^{m-1}\frac{P_{m-1-i}}{(m-1-i)!}\frac{\mathbb{E} \varphi^i(f)}{i!}(m-1)!\right)=(m)P_{m-1}(x);$$
		$$\mathbb{E}P_m(x)=\mathbb{E}\left(x^{m} - \sum\limits_{i=1}^{m}\frac{P_{m-i}}{(m-i)!}\frac{\mathbb{E} \varphi^i(f)}{i!}(m)!\right) = \mathbb{E}\left(\varphi^m(f)-\mathbb{E}\varphi^m(f)\right) = 0.$$
		
		All properties given in the definition of Wick's powers \ref{def_Wick_power} are fulfilled. So, we get:
		$$\frac{\exp(\lambda \varphi(f))}{\mathbb{E} \exp(\lambda \varphi(f))} = \sum\limits_{m=0}^{\infty}\frac{\lambda^m :\varphi^m(f):}{m!},$$
		$$:\varphi^{m}(f):~= \frac{d^m}{d \lambda^m}\left[\frac{\exp\left(\lambda \varphi(f)\right)}{E \exp\left(\lambda \varphi(f)\right)}\right]$$

		Analogously to proof of the lemma \ref{lem_moments_powers}, one can show that
		$$E e^{\lambda \varphi(f)} = \exp\left(\sum\limits_{k=1}^n \sum_{p = 1}^\infty (-1)^{p+1} \frac{|B_{2p}|2^{2p}\left(2^{2p} -1\right)}{2p(2p)!} \left( \lambda f \left(\frac{k}{n}\right)\frac{1}{\sqrt{n}}\right)^{2p}\right)$$
		
		Hence 
		$$\frac{\exp\left(\lambda \varphi(f)\right)}{E \exp\left(\lambda \varphi(f)\right)} = $$
		$$ = \exp \left[\lambda \sum\limits_{k=1}^n f \left(\frac{k}{n}\right)\frac{\varepsilon_k}{\sqrt{n}} + \sum\limits_{k=1}^n \sum_{p = 1}^\infty (-1)^p \frac{|B_{2p}|2^{2p}\left(2^{2p} -1\right)}{2p(2p)!} \left( \lambda f \left(\frac{k}{n}\right)\frac{1}{\sqrt{n}}\right)^{2p}\right]$$
		
		Using Faa di Bruno's formula for the derivatives of the composition of two real-valued functions \cite{Bruno} one can get the final expression:
		$$:\varphi^{m}(f):~= \sum\limits_{\pi \in \Pi} \prod\limits_{D \in \pi} \Bigg[\varphi^{|D|}(f) \lambda^{1-|D|} \mathbbm{1}_{(1-|D| \geqslant 0)} + $$
		$$ + (-1)^{\frac{|D|}{2}} b_p \frac{1}{(2p-|D|)!}\sum\limits_{k=1}^n \left(f\left(\frac{k}{n}\right) \frac{1}{\sqrt{n}}\right)^{|D|} \left( \lambda f\left(\frac{k}{n}\right) \frac{1}{\sqrt{n}}\right)^{2p-|D|} \mathbbm{1}_{(2p-|D| \geqslant 0)} \Bigg]_{\lambda = 0},$$
		where  $\Pi$ is a set of all possible partitions of a set $\left\{1_1, 1_2, \ldots, 1_m\right\}.$ and $b_p = \frac{|B_{2p}|2^{2p}\left(2^{2p} -1\right)}{2p}.$
	\end{proof}

	\section{Feynman's diagrams for Bernoulli case}
	
	Lemma \ref{lem_moments_powers} allows us to calculate the moments for the products of random variables from the Bernoulli noise. However, it is useful to be able to calculate moments for more complex functions from the Bernoulli noise, such as Wick's products. For this purpose, in the case of Gaussian white noise, we can use Feynman diagrams. Let's construct an analog of Feynman diagrams for Bernoulli noise.
	
	Consider a random variable 
	$$A(\varphi) := \prod\limits_{i=1}^N :\varphi^{n_i}(f_i):$$
	
	To calculate an expectation for $A(\varphi),$ construct a graph where each multiplier ${:\varphi^{n_i}(f_i):}$ is assigned to a set of vertices $\left\{f_i^1, \ldots, f_i^{n_i}\right\}.$
	
	Define a bijective mapping that assigns an individual number to each vertex:
	$$\psi : F \to \left\{1, 2, \ldots, K\right\},$$
	where $F = \left\{f_{i}^j, ~i=\overline{1,N},\forall i: j=\overline{1,n_i}\right\}$ is a set of all vertices and $|F| = K = \sum\limits_{i=1}^N n_i$ is the number of vertices in the diagram.

	Let's build a diagram $G$ according to the following rules:
	\begin{enumerate}
		\item Construct a set $\Pi$ of all partition of a set $\left\{1, 2, \ldots, K\right\}$ where each block has an even cardinality.
		Partition with blocks that include a set containing only vertices corresponding to the same multiplier $:\varphi^{n_i}(f_i):$ are not taken into account.
		
		\item For each block $s$ of a partition $ S \in \Pi $ we build a method of traversing through all its vertices, which does not pass along the same edge more than once, taking into account the direction of traversal. The traversing starts from the vertex of the graph with the smallest ordinal number. That is, for each partition element $ S $ we choose the vertex with the smallest number:
		$$s_1 = \min\limits_{s} s_i,$$
		and the corresponding traversal method will be specified by an ordered set of vertices:
		$$\left\{s_1\right\} \cup \sigma(s\backslash\left\{s_1\right\}),$$
		where $ \sigma (\cdot) $ is permutation.
		Therefore, the diagram $G$ will correspond to one of the options for traversing blocks from one of the partitions.
	\end{enumerate}

	\begin{theorem}[An analogue of Feynman diagrams for Bernoulli noise]\label{theorem_diagram_Bernuolli}
		\begin{gather}\label{3}
			E A(\varphi) = \sum\limits_{\left\{G\right\}} I(G), \text{where} \nonumber \\
			I(G) = \prod\limits_{s \in S} (-1)^{m-1} \sum\limits_{k=1}^n \prod\limits_{i: \psi(f_i) \in s} f_i\left(\frac{k}{n}\right) \frac{1}{\sqrt{n}},
		\end{gather}
		where $\left\{G\right\}$ is a set of all possible diagrams, $S$ is a partition in which every element $s \in S$ arranged in accordance with the order and direction of the traversing, $m $ is a number of pairs $(a_i, a_{i+1})$ of two neighboring elements from $s$ where $a_i \leqslant a_{i+1}.$
	\end{theorem}

\begin{example}\label{ex_diagram_wick_bernuolli}
	Consider
	 $$A(\varphi) = :\varphi^2(f_1): \cdot :\varphi^2(f_2):$$
	 
	Let's define a set of vertices $f_{1_1}, f_{1_2}, f_{2_1}, f_{2_2}.$ 
	Give each vertex an individual number:
	$$\psi(f_{1_1}) = 1,~\psi(f_{1_2}) = 2,~\psi(f_{2_1}) = 3,~\psi(f_{2_2}) = 4.$$
	
	The set  of all the partitions of a set $\left\{1, 2, 3, 4\right\}$ looks like
	$$\Pi = \left\{ \left\{\left\{1,3\right\},\left\{2,4\right\}\right\},\left\{\left\{1,4\right\},\left\{2,3\right\}\right\},\left\{\left\{1,2,3,4\right\}\right\}\right\}.$$
	
	Let's construct diagrams using this set, according to the rules mentioned above. Consider a first partition $\pi_1 = \left\{\left\{1,3\right\},\left\{2,4\right\}\right\}.$ Only one diagram $G_1$ will correspond to this partition. Hence, $S_1 = \pi_1 = \left\{\left\{1,3\right\},\left\{2,4\right\}\right\}, ~m_{1_1} = 1,$ (since $1 \leqslant 3$),  $m_{1_2} = 1, $ (since $2 \leqslant 4$).
	
	Similarly, diagram $G_2$ will correspond to  $S_2 = \pi_2 = \left\{\left\{1,4\right\},\left\{2,3\right\}\right\}, ~m_{2_1} = 1,$ (since $1 \leqslant 4$), $m_{2_2} = 1, $ (since $2 \leqslant 3$).
	
	Diagrams $G_1$ and $G_2$ will look as follows
	
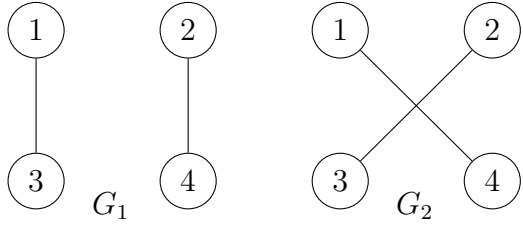
\begin{center}
		\begin{tikzpicture}[node distance={20mm},main/.style = {draw, circle}] 
			\node[main] (1) {$1$}; 
			\node[main] (2) [right of=1] {$2$}; 
			\node[main] (3) [below of=1] {$3$};
			\node[main] (4) [below of=2] {$4$};
			
			\node[main] (5) [right of=2] {$1$}; 
			\node[main] (6) [right of=5] {$2$}; 
			\node[main] (7) [below of=5] {$3$};
			\node[main] (8) [below of=6] {$4$};
			
			\draw (1) to (3) node[below right]{$~~~~G_1$};
			\draw (2) to (4);
			\draw (5) to (8);
			\draw (6) to (7) node[below right]{$~~~~G_2$};		
			
		\end{tikzpicture}
	\end{center}

	Determine which terms correspond to the diagrams $G_1$ and $G_2$:
	$$I(G_1) = (-1)^0 \sum\limits_{k=1}^n f_1\left(\frac{k}{n}\right) \frac{1}{\sqrt{n}} ~ f_2\left(\frac{k}{n}\right) \frac{1}{\sqrt{n}} \cdot (-1)^0 \sum\limits_{k=1}^n f_1\left(\frac{k}{n}\right) \frac{1}{\sqrt{n}}~ f_2\left(\frac{k}{n}\right) \frac{1}{\sqrt{n}} =$$
	$$ = \left(\sum\limits_{k=1}^n f_1\left(\frac{k}{n}\right) f_2\left(\frac{k}{n}\right) \frac{1}{n} \right)^2 ,$$
	$$I(G_2) = (-1)^0 \sum\limits_{k=1}^n f_1\left(\frac{k}{n}\right) \frac{1}{\sqrt{n}} ~ f_2\left(\frac{k}{n}\right) \frac{1}{\sqrt{n}} \cdot (-1)^0 \sum\limits_{k=1}^n f_1\left(\frac{k}{n}\right) \frac{1}{\sqrt{n}}~ f_2\left(\frac{k}{n}\right) \frac{1}{\sqrt{n}} =$$
	$$ = \left(\sum\limits_{k=1}^n f_1\left(\frac{k}{n}\right) f_2\left(\frac{k}{n}\right) \frac{1}{n} \right)^2.$$
	
	Now consider the partition $\pi_3 = \left\{\left\{1,2,3,4\right\}\right\}.$ This partition will correspond to six different options for traversing through vertices
	
	$$S_3 = \left\{\left\{1,2,3,4\right\}\right\}, m_3 = 3  ~ (1<2,~ 2<3,~ 3<4 ),$$
	
	$$S_4 = \left\{\left\{1,4,3,2\right\}\right\}, m_4 = 1 ~ (1<4 ),$$
	
	$$S_5 = \left\{\left\{1,3,4,2\right\}\right\}, m_5 = 2 ~ (1<3,~ 3<4 ),$$
	
	$$S_6 = \left\{\left\{1,2,4,3\right\}\right\}, m_6 = 2 ~ (1<2,~ 2<4),$$
	
	$$S_7 = \left\{\left\{1,3,2,4\right\}\right\}, m_7 = 2 ~ (1<3,~ 2<4 ),$$
	
	$$S_8 = \left\{\left\{1,4,2,3\right\}\right\}, m_8 = 2 ~ (1<4,~ 2<3) .$$
	
	The corresponding diagrams will look like this:
	
	
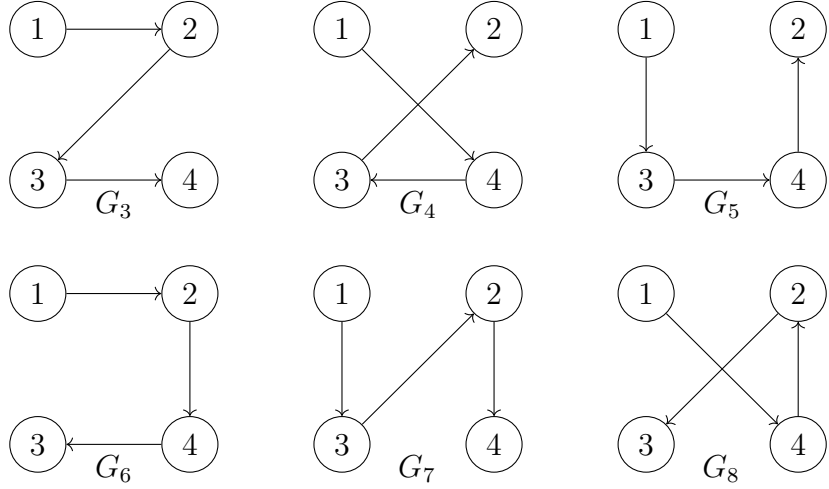
\begin{center}
		\begin{tikzpicture}[node distance={20mm},main/.style = {draw, circle}] 
			\node[main] (1) {$1$}; 
			\node[main] (2) [right of=1] {$2$}; 
			\node[main] (3) [below of=1] {$3$};
			\node[main] (4) [below of=2] {$4$};
			
			\node[main] (5) [right of=2] {$1$}; 
			\node[main] (6) [right of=5] {$2$}; 
			\node[main] (7) [below of=5] {$3$};
			\node[main] (8) [below of=6] {$4$};
			
			\node[main] (9) [right of=6] {$1$}; 
			\node[main] (10) [right of=9] {$2$}; 
			\node[main] (11) [below of=9] {$3$};
			\node[main] (12) [below of=10] {$4$};
			
			\draw[->] (1) to (2);
			\draw[->] (2) to (3);
			\draw[->] (3) to node[midway, below] {$G_3$} (4);
			
			\draw[->] (5) to (8);
			\draw[->] (8) to node[midway, below] {$G_4$} (7);
			\draw[->] (7) to (6);
			
			\draw[->] (9) to (11);
			\draw[->] (11) to node[midway, below] {$G_5$} (12);
			\draw[->] (12) to (10);
			
		\end{tikzpicture}
	\end{center}
	\begin{center}
		\begin{tikzpicture}[node distance={20mm},main/.style = {draw, circle}] 
			\node[main] (1) {$1$}; 
			\node[main] (2) [right of=1] {$2$}; 
			\node[main] (3) [below of=1] {$3$};
			\node[main] (4) [below of=2] {$4$};
			
			\node[main] (5) [right of=2] {$1$}; 
			\node[main] (6) [right of=5] {$2$}; 
			\node[main] (7) [below of=5] {$3$};
			\node[main] (8) [below of=6] {$4$};
			
			\node[main] (9) [right of=6] {$1$}; 
			\node[main] (10) [right of=9] {$2$}; 
			\node[main] (11) [below of=9] {$3$};
			\node[main] (12) [below of=10] {$4$};
			
			\draw[->] (1) to (2);
			\draw[->] (2) to (4);
			\draw[->] (4) to node[midway, below] {$G_6$} (3);
			
			\draw[->] (5) to (7) node[below right]{$~~~~G_7$};
			\draw[->] (7) to (6);
			\draw[->] (6) to (8);
			
			\draw[->] (9) to (12);
			\draw[->] (12) to (10);
			\draw[->] (10) to (11) node[below right]{$~~~~G_8$};
		\end{tikzpicture}
	\end{center}

	Determine which term corresponds to each diagram:
	$$I(G_3) = I(G_4) = \sum\limits_{k=1}^n f_1^2 \left(\frac{k}{n}\right) f_2^2 \left(\frac{k}{n}\right) \frac{1}{n},$$
	$$I(G_5) = I(G_6) = I(G_7) = I(G_8) = -\sum\limits_{k=1}^n f_1^2 \left(\frac{k}{n}\right) f_2^2 \left(\frac{k}{n}\right) \frac{1}{n}.$$
	\par Hence, 
	$$E\left[:\varphi^2(f_1)::\varphi^2(f_2):\right] = \sum\limits_{\left\{G\right\}} I(G) = $$
	$$ = 2\left(\sum\limits_{k=1}^n f_1\left(\frac{k}{n}\right) f_2\left(\frac{k}{n}\right) \frac{1}{n} \right)^2 - 2\sum\limits_{k=1}^n f_1^2 \left(\frac{k}{n}\right) f_2^2 \left(\frac{k}{n}\right) \frac{1}{n}. $$
	
\end{example}

Now, let's prove theorem \ref{theorem_diagram_Bernuolli}.
\begin{proof} (An analog of Feynman diagrams for Bernoulli noise)

	Using the formula (2) one can get:
	\begin{gather}\label{4}
		\mathbb{E} \left[:\varphi^{n_1}(f_1): \cdot \ldots \cdot :\varphi^{n_N}(f_N):\right] = \nonumber \\
		= \mathbb{E}\Bigg[\Bigg(\sum\limits_{\pi \in \Pi_1} \prod\limits_{D \in \pi} \Big[\varphi^d(f_1) \lambda^{1-|D|} \mathbbm{1}_{(1-|D| \geqslant 0)} + \nonumber \\
		(-1)^{\frac{|D|}{2}} b_p \frac{1}{(2-|D|)!}\sum\limits_{k=1}^n f_1^d\left(\frac{k}{n}\right) \frac{1}{\sqrt{n}} \left( \lambda f\left(\frac{k}{n}\right) \frac{1}{\sqrt{n}}\right)^{2p-|D|} \mathbbm{1}_{(2p-|D| \geqslant 0)} \Big]\Bigg) \cdot  \nonumber \\
		\cdot \ldots \cdot \Bigg(\sum\limits_{\pi \in \Pi_N} \prod\limits_{D \in \pi} \Big[\varphi^d(f_N) \lambda^{1-|D|} \mathbbm{1}_{(1-|D| \geqslant 0)} + \nonumber \\
		(-1)^{\frac{|D|}{2}} b_p \frac{1}{(2-|D|)!}\sum\limits_{k=1}^n f_N^d\left(\frac{k}{n}\right) \frac{1}{\sqrt{n}} \left( \lambda f\left(\frac{k}{n}\right) \frac{1}{\sqrt{n}}\right)^{2p-|D|} \mathbbm{1}_{(2p-|D| \geqslant 0)} \Big]\Bigg)\Bigg]
	\end{gather}

	After multiplication, we will get the sum of all possible combinations of the partitions of the sets  $\left\{1_1, 1_2, \ldots, 1_{n_1}\right\}, \ldots, \left\{N_1, N_2, \ldots, N_{n_N}\right\}.$
	
	Calculating the expectation, all possible terms of the following form will occur:
	$$ \prod\limits_{s \in S} \sum\limits_{k=1}^n \prod\limits_{i \in s} f_i\left(\frac{k}{n}\right) \frac{1}{\sqrt{n}},$$ 
	where $S$ is a partition of $\left\{1_1, \ldots, 1_{n_1}, 2_1, \ldots, 2_{n_2}, \ldots, N_{n_N}\right\}$ on the blocks with even cardinality.

	Consider the term
	\begin{gather}\label{5}
		\prod\limits_{j=1}^N \prod\limits_{i=1}^{I_j} \sum\limits_{k=1}^n f_j^{2k_{i,j}} \left(\frac{k}{n}\right) \frac{1}{\sqrt{n}} ~\cdot~ \prod\limits_{s \in S'} \sum\limits_{k=1}^n \prod\limits_{t \in s} f_t \left(\frac{k}{n}\right) \frac{1}{\sqrt{n}},
	\end{gather}
	with no terms of the form $\sum\limits_{k=1}^n f_i^m\left(\frac{k}{n}\right) \left(\frac{1}{\sqrt{n}}\right),~m > 1$ in the second part.
	Let us determine when calculating the mathematical expectation from which terms from (4), the term (5) can appear. Each such term corresponds to a certain combination of partitions. Since there are no sums of the form  $\sum\limits_{k=1}^n f_i^m\left(\frac{k}{n}\right) \left(\frac{1}{\sqrt{n}}\right),~m > 1$ in the second part, this term can only appear when calculating the mathematical expectation $\prod\limits_{s \in S'} \prod\limits_{t \in s} \varphi(f_t).$
	
	Let's fix a certain partition R for $:\varphi^{n_2}(f_2):, \ldots, :\varphi^{n_N}(f_N):~$, which gives, after calculating the mathematical expectation, the term (5). That is:
	$$R = R_2 \cup R_3 \cup \ldots \cup R_N, ~\text{where } R_i \in \Pi_i.$$

	Now consider all possible partition $R_1$ for $:\varphi^{n_1}(f_1):~$, which corresponds to the term~(5).
	
	There will be $I_1$ sums of the form $\sum\limits_{k=1}^n f_1^{2k_{1,j}} \left(\frac{k}{n}\right) \frac{1}{\sqrt{n}}$ in (5), and each of these sums could appear as a constant from (4), that corresponds to the partition $\left\{\left\{1, \ldots, 1\right\}\right\} $ or as an expectation of $\varphi^{2k_{1,j}}(f_1)$, that corresponds to the partition $\left\{\left\{1\right\}, \ldots \left\{1\right\}\right\}.$ 
	
	Let's define and numerate the types of partitions according to the number of sums, represented by the constant:
	
	0. $\left\{\left\{1\right\},\left\{1\right\}, \ldots \left\{1\right\}\right\} \cup R;$\\
	
	1. $\left\{\{\underbrace{1, \ldots, 1}_{2k_{1,i}}\} , \left\{1\right\},\left\{1\right\}, \ldots \left\{1\right\}\right\} \cup R;$\\
	
	2. $\left\{\{\underbrace{1, \ldots, 1}_{2k_{1,i}}\} , \{\underbrace{1, \ldots, 1} _{2k_{1,j}}\}, \left\{1\right\},\left\{1\right\}, \ldots \left\{1\right\}\right\} \cup R;$
	
	$\vdots$\\
	
	$I_1. \left\{\{\underbrace{1, \ldots, 1}_{2k_{1,i}}\}, \ldots, \{\underbrace{1, \ldots, 1} _{2k_{1,I_1}}\}, \left\{1\right\},\left\{1\right\}, \ldots \left\{1\right\}\right\} \cup R.\\$
	
	Now let's find the coefficient for (5) in the expression for\\
	$\mathbb{E} \left[:\varphi^{n_1}(f_1): \cdot \ldots \cdot:\varphi^{n_N}(f_N):\right]$ for fixed $R.$ To do this, we first find the coefficient with which (5) will appear when calculating the mathematical expectation of each type of the term and then sum it up.
	
	0. This partition in the formula (4) corresponds to the term \\ $\varphi^{n_1}(f_1)P_R(\varphi(f_2), \ldots \varphi(f_N))$, where $P_R$ is a is some polynomial corresponding to the partition $R.$ There will be $Q_{P_R}$ such terms in (4), where $Q_{P_R}$ is a constant appearing in (2). Let's determine the coefficient for the term (5) when calculating the mathematical expectation from $Q_{P_R}\varphi^{n_1}(f_1) P_R (\varphi(f_2), \ldots, \varphi(f_N)).$ The following partition will correspond to the term(\ref{5}) in the formula (1):
	$$\{\underbrace{\left\{1, \ldots, 1\right\} }_{2k_{1,i}}, \ldots, \underbrace{\left\{1, \ldots, 1\right\} }_{2k_{1,I_1}}, \} \cup S'.$$
	
	Then the coefficient is:
	$$a_0 = Q_{P_R} \cdot Q_{S'} \cdot C_{n_1}^{2k_{1,1}} \cdot C_{n_1 - 2k_{1,1}}^{2k_{1,2}} \cdot \ldots \cdot C_{n_1 - \sum_{i=1}^{I_1 - 1} 2k_{1,i}}^{2k_{1,I_1}} \cdot \prod\limits_{i=1}^{I_1} \left((-1)^{k_{1,i} + 1} b_{|2k_{1,i}|}\right),$$
	where $Q_{S'} $ is a coefficient corresponding to the partition $S'$ in the formula (1), \\ $C_{n_1 - 2k_{1,1}}^{2k_{1,2}} \cdot \ldots \cdot C_{n_1 - \sum_{i=1}^{I_1 - 1} 2k_{1,i}}^{2k_{1,I_1}} $the number of options to group elements in $I_1$ sets with cardinalities $2k_{1,1} \ldots 2k_{1,I_1},$ and $b_{|2k_{1,i}|} $ is a coeficient corresponding to $\{\underbrace{1, \ldots, 1}_{2k_{1,i}}\} $ in formula~(1). Denote 
	$$Q := Q_{P_R} \cdot Q_{S'} \cdot C_{n_1}^{2k_{1,1}} \cdot C_{n_1 - 2k_{1,1}}^{2k_{1,2}} \cdot \ldots \cdot C_{n_1 - \sum_{i=1}^{I_1 - 1} 2k_{1,i}}^{2k_{1,I_1}} \cdot \prod\limits_{i=1}^{I_1} \left( b_{|2k_{1,i}|}\right).$$
	
	Then 
	$$a_0 = Q (-1)^{\sum\limits_{i=1}^{I_1} k_{1,i}~ + I_1}. $$
	
	1. Let's fix $i$. This partition corresponds to
	
	$$\sum\limits_{k=1}^n f_1^{2k,i} \left(\frac{k}{n}\right) \frac{1}{\sqrt{n}} \cdot \varphi^{n_1 - 2k_{1,i}}(f_1) \cdot P_R(\varphi(f_2), \ldots \varphi(f_N)).$$
	
	There will be $$Q_{P_R} \cdot (-1)^{k_{1,i}} b_{|2k_{1,i}|} \cdot C_{n_1}^{2k_{1,i}}$$ such terms in (5),
	where $ (-1)^{k_{1,i}} b_{|2k_{1,i}|} \cdot C_{n_1}^{2k_{1,i}}$ is a coefficient from (2) which corresponds to~$\{\underbrace{1, \ldots, 1}_{2k_{1,i}}\}.$
	
	Now let's determine the coefficient for (5) in expectation of
	$$Q_{P_R} \cdot (-1)^{k_{1,i}} b_{|2k_{1,i}|} \cdot C_{n_1}^{2k_{1,i}} \sum\limits_{k=1}^n f_1^{2k,i} \left(\frac{k}{n}\right) \frac{1}{\sqrt{n}} \cdot \varphi^{n_1 - 2k_{1,i}}(f_1) \cdot P_R(\varphi(f_2), \ldots \varphi(f_N)).$$
	
	The following term will correspond to the term (5) in expression (1):
	$$\{\underbrace{\left\{1, \ldots, 1\right\} }_{2k_{1,i}}, \ldots, \underbrace{\left\{1, \ldots, 1\right\} }_{2k_{1,I_1}}, \} \cup S'.$$
	
	Therefore, the coefficient will be as follows:
	$$a_1^i = Q_{P_R} \cdot (-1)^{k_{1,i}} b_{|2k_{1,i}|} \cdot C_{n_1}^{2k_{1,i}} \cdot \prod\limits_{j=1, j \neq i}^{I_1} \left((-1)^{k_{1,j} + 1} b_{|2k_{1,j}|}\right) \cdot C_{n_1 - 2k_{1,i}}^{2k_{1,2}} \cdot \ldots \cdot C_{n_1 - \sum_{i=1}^{I_1 - 1} 2k_{1,i}}^{2k_{1,I_1}} = $$
	$$ = (-1)^{\sum\limits_{i=1}^{I_1} k_{1,i}~ + I_1 -1} \cdot Q.$$
	
	As we can see, the coefficient $a_1^i $ does not depend on $i$. Given that $C_{I_1}^1 $ is the number of options to choose $i$, we get:
	$$a_1 = Q C_{I_1}^1(-1)^{\sum\limits_{i=1}^{I_1} k_{1,i}~ + I_1 - 1} $$
	
	Similarly, we get:	 
	
	$$a_m = Q C_{I_1}^m (-1)^{\sum\limits_{i=1}^{I_1} k_{1,i}~ + I_1 - m}, ~ m=\overline{1,I_1}$$
	
	Therefore, the general coefficient for the term (\ref{5}) in the expression for\\
	$\mathbb{E} \left[:\varphi^{n_1}(f_1): \cdot \ldots \cdot:\varphi^{n_N}(f_N):\right]$ for fixed $R$ will be equal:
	
	$$A = \sum\limits_{m=0}^{I_1} a_m = Q \sum\limits_{m=0}^{I_1} C_{I_1}^m (-1)^{\sum\limits_{i=1}^{I_1} k_{1,i}~ + I_1 - m} = 0$$
	
	So, if the term contains a sum of the form $\sum\limits_{k=1}^{n} f_i^m \left(\frac{k}{n}\right) \frac{1}{\sqrt{n} },$ then such a term will appear in $\mathbb{E} \left[:\varphi^{n_1}(f_1): \cdot \ldots \cdot:\varphi^{n_N}(f_N):\right]$ with zero coefficient. This confirms the assumption that traversals in the diagrams cannot be performed on vertices corresponding to the same $f_i.$
	In this case, $\mathbb{E} \left[:\varphi^{n_1}(f_1): \cdot \ldots \cdot:\varphi^{n_N}(f_N):\right]$ will be determined by terms from $\mathbb{E} \left[\varphi^{n_1}(f_1) \cdot \ldots \cdot \varphi^{n_N}(f_N)\right],$ where there are no sums of the form $\sum\limits_{k=1}^{n} f_i^m \left(\frac{k}{n}\right) \frac{1}{\sqrt{n}}.$
	
	Hence,
	\begin{gather}\label{6}
		\mathbb{E} \left[:\varphi^{n_1}(f_1): \cdot \ldots \cdot:\varphi^{n_N}(f_N):\right] = \sum\limits_{\Pi'} \prod\limits_{D \in \Pi}   (-1)^{\frac{|D|}{2}}b_{|D|} \sum\limits_{k=1}^n \prod\limits_{d \in D} \left(f_d\left(\frac{k}{n}\right)\frac{1}{\sqrt{n}}\right)
	\end{gather}
	where $b_{|D|} =  \frac{ B_{|D|} 2^{|D|} (2^{|D|} - 1)}{|D|}$,  $\Pi'$ is a set of all possible partitions of the set $\left\{1_1, 1_2, \ldots 1_{n_1},  \ldots  N_{n_N}\right\} $ on blocks with even cardinality, which does not contain sets $\left\{i_{j_1}, \ldots i_{j_l}\right\}.$ 
	
	Let's prove that (6) is equal to (3).
	
	Consider a factor from (\ref{3}) corresponding to a certain diagram $G$ and a certain element $s \in S$ 
	$$s = \left\{s_1, s_2, \ldots, s_{2t}\right\}$$
	
	At the same time, $$ \forall i \neq 1:s_1 < s_i,$$ since the traversal starts from the vertex with the smallest number.
	
	$\left\{\sigma(s)\right\}$ is the set of all permutations $\left\{s_2, \ldots , s_{2t}\right\},$
	
	$l = m-1 $ is the number of "ascents" in $\left\{s_2, \ldots, s_{2t}\right\}.$
	
	Let $F_s = \left\{f_i^{r_j}: \psi(f_i^{r_j}) \in s\right\}.$ It can be noted that $F_s = F_{\left\{s_1\right\} \cup \sigma(s)},$ therefore all sets of the form $\left\{s_1\right\} \cup \sigma(s)$ generate the same factor
	$$\sum\limits_{k=1}^n \prod_{f_i \in F_s} f_i \left(\frac{k}{n}\right) \frac{1}{\sqrt{n}},$$
	the difference will be only in the coefficient $(-1)^l$, which is determined by the number of ascents in $\sigma(s).$
	
	To find the coefficient for the factor $$\sum\limits_{k=1}^n \prod_{f_i \in F_s} f_i \left(\frac{k}{n}\right) \frac{1}{\sqrt{n}}$$ in the expression for $E \left[:\varphi^{n_1}(f_1): \cdot \ldots \cdot:\varphi^{n_N}(f_N):\right]$ we need to sum the coefficients corresponding to all the permutations  $\sigma(s)$:
	$$\sum\limits_{l=0}^{|s| - 1} (-1)^{l}E(|s| - 1, l) = \frac{ B_{|s|} 2^{|s|} (2^{|s|} - 1)}{|s|},$$
	
	where $E(|s| - 1, l) $ is an Eulerian number, which is a number of permutations of $|s| - 1$ elements with $l $ ascents, and $B_{|s|}$ are Bernoulli numbers \cite{Euler}.

	Define $D = \left\{i_j: \psi(f_i^j) \in s\right\}.$ Since $\psi $ is a bijective mapping, therefore $|D| = |s|.$
	
	So, applying (3), we got the same coefficient at $$\sum\limits_{k=1}^n \prod_{d \in D} f_d\left(\frac{n}{k }\right) \frac{1}{\sqrt{n}},$$ as when applying (6).

\end{proof}

\section{Limit behavior of Feynman diagrams for Bernoulli noise}
	
	Bernoulli noise was constructed as an approximation of Gaussian white noise. To prove this, let's show the convergence of elements of Bernoulli noise to corresponding elements of Gaussian white noise. 
	
	\begin{theorem}[Weak convergence to Gaussian white noise]\label{theorem_convergence_bernuolli_to_gaussian}
		
		$\forall f \in C\left(\left[0,1\right]\right):$
		
		$$\varphi(f) \Longrightarrow N(0, \|f\|^2), n \to \infty $$
	\end{theorem}
	\begin{proof}
		We will show that the characteristic function of $\varphi(f)$ converges to the characteristic function of a random variable with normal distribution $N(0, \|f\|^2):$
		
		$$\mathbb{E} e^{it\varphi(f)} = \mathbb{E} \exp\left(it\sum\limits_{k=0}^n f\left(\frac{k}{n}\right)\frac{\varepsilon_k}{\sqrt{n}}\right) = \prod\limits_{k=1}^n \mathbb{E} \exp\left(it f\left(\frac{k}{n}\right)\frac{\varepsilon_k}{\sqrt{n}}\right) = $$
		
		$$ = \prod\limits_{k=1}^n \frac{1}{2} \left(e^{itf\left(\frac{k}{n}\right)\frac{1}{\sqrt{n}}} + e^{-itf\left(\frac{k}{n}\right)\frac{1}{\sqrt{n}}}\right) = \prod\limits_{k=1}^n \cos\left(tf\left(\frac{k}{n}\right)\frac{1}{\sqrt{n}}\right) = $$
		$$ = \exp\left(\sum\limits_{k=0}^n \ln \cos \left(tf\left(\frac{k}{n}\right)\frac{1}{\sqrt{n}}\right)\right)$$
		
		Noting that $ln\left(cos(x)\right) = - \sum_{p = 1}^\infty b_p x^{2p},$ where $b_p = \frac{|B_{2p}|2^{2p}\left(2^{2p} -1\right)}{2p(2p)!}$ and $B_{2p} $ are Bernoulli numbers, one can get:
		$$\exp\left(\sum\limits_{k=0}^n \ln \cos \left(tf\left(\frac{k}{n}\right)\frac{1}{\sqrt{n}}\right)\right) = \exp\left(-\sum\limits_{k=0}^n \sum\limits_{p=1}^\infty b_p \left(t f\left(\frac{k}{n}\right)\frac{1}{\sqrt{n}}\right)^{2p}\right)$$
		
		Consider the term when $p \geqslant 2.$ Since $f$ is continuous on the closed interval, it is bounded:
		$$\exists M \in \mathbb{R}: \forall x \in [0,1]: ~ |f(x)| \leqslant M.$$
		
		Taking into account that fact:
		$$\sum\limits_{k=0}^n b_p t^{2p} f^{2p}\left(\frac{k}{n}\right)\frac{1}{n^p} \leqslant \sum\limits_{k=0}^n b_p t^{2p} M^{2p} \frac{1}{n^p} = b_p t^{2p} M^{2p} \frac{1}{n^{p-1}} \longrightarrow 0, n \to \infty.$$

		When $p=1$ the term $\sum\limits_{k=0}^n  \frac{1}{2} f^2\left(\frac{k}{n}\right)\frac{1}{n} $ is a Riemann sum for integral
		$$\frac{1}{2}\int\limits_0^1 f^2(s)ds.$$
		
		Hence:
		$$ e^{it\varphi(f)} \longrightarrow e^{-\frac{1}{2} t^2 \|f\|^2},$$
		$$\varphi(f) \Longrightarrow N(0, \|f\|^2), n \to \infty.  $$
	\end{proof}
	
	\begin{remark}\label{vector_convergence}
	    Similarly, it can be shown that
	    $$\left( \varphi(f_1), \ldots, \varphi(f_k) \right) \longrightarrow \left( f_1(\xi), \ldots, f_k (\xi) \right),~n \to \infty. $$
	\end{remark}

	We can expect that Feynman diagrams for Bernoulli noise will converge to the Feynman diagrams for Gaussian case~\cite{reed_simon1}.

	\begin{lemma}
		Consider a diagram $G $ and corresponding partition $S.$ If there exists such  $s \in S,$ that $|s| > 2,$ then $$I(G) \longrightarrow 0, n \to \infty $$
	\end{lemma} 
	\begin{proof}
		Suppose there exists such $s \in S,$ that $|s| > 2.$ Then corresponding multiplier in $I(G),$ according to formula \ref{2}, equals to
		$$(-1)^{m-1} \sum\limits_{k=1}^n \prod\limits_{i: \psi(f_i) \in s} f_i\left(\frac{k}{n}\right) \frac{1}{\sqrt{n}}$$
		
		 $$\forall i: ~ f_i \in C([0,1]),$$
		so $$\forall i ~ \exists M_i \in \mathbb{R}: \forall x \in [0,1]: |f_i(x)| \leqslant M_i$$
		
		Consider
		$$M = \max\limits_{i} M_i$$
		
		Then
		$$\left|(-1)^{m-1} \sum\limits_{k=1}^n \prod\limits_{i: \psi(f_i) \in s} f_i\left(\frac{k}{n}\right) \frac{1}{\sqrt{n}}\right| \leqslant \sum\limits_{k=1}^n \prod\limits_{i: \psi(f_i) \in s} \left|f_i\left(\frac{k}{n}\right)\right| \frac{1}{\sqrt{n}} \leqslant $$
		$$\leqslant \sum\limits_{k=1}^n M^{|s|} \frac{1}{n^{\frac{|s|}{2}}} = M^{|s|} \frac{1}{n^{\frac{|s|}{2} - 1}} \longrightarrow 0, ~n \to \infty$$
		
		Hence
		$$I(G) = \prod\limits_{s \in S} (-1)^{m-1} \sum\limits_{k=1}^n \prod\limits_{i: \psi(f_i) \in s} f_i\left(\frac{k}{n}\right) \frac{1}{\sqrt{n}} \longrightarrow 0, ~n \to \infty.$$
	\end{proof}

	Therefore, when $n$ goes to infinity, only diagrams with partition into pairs will have a non-zero impact on the expectation. This exactly describes Feynman diagrams for Gaussian random variables. So, we get
	$$\mathbb{E} \left[:\varphi^{n_1}(f_1): \cdot \ldots \cdot :\varphi^{n_N}(f_N):\right] \longrightarrow  \mathbb{E} \left[:(f_1, \xi)^{n_1}: \cdot \ldots \cdot :(f_N, \xi)^{n_N}:\right], n \to \infty.$$

	\begin{example}\label{ex_diagram_wick_bernuolli_to_gaussian}
		From the example \ref{ex_diagram_wick_bernuolli} for
		$$\mathbb{E}\left[:\varphi^2(f_1)::\varphi^2(f_2):\right]$$
		we got the following diagrams:
		\begin{figure}[h!]
			\centering
			\begin{tikzpicture}[node distance={20mm},main/.style = {draw, circle}] 
				\node[main] (1) {$1$}; 
				\node[main] (2) [right of=1] {$2$}; 
				\node[main] (3) [below of=1] {$3$};
				\node[main] (4) [below of=2] {$4$};
				
				\node[main] (5) [right of=2] {$1$}; 
				\node[main] (6) [right of=5] {$2$}; 
				\node[main] (7) [below of=5] {$3$};
				\node[main] (8) [below of=6] {$4$};
				
				\node[main] (9) [right of=6]{$1$}; 
				\node[main] (10) [right of=9] {$2$}; 
				\node[main] (11) [below of=9] {$3$};
				\node[main] (12) [below of=10] {$4$};
				
				\node[main] (13) [right of=10] {$1$}; 
				\node[main] (14) [right of=13] {$2$}; 
				\node[main] (15) [below of=13] {$3$};
				\node[main] (16) [below of=14] {$4$};
				
				\draw (1) to (3) node[below right]{$~~~G_1$};
				\draw (2) to (4);
				
				\draw (5) to (8);
				\draw (6) to (7) node[below right]{$~~~G_2$};	
				
				\draw[->, dashed] (9) to (10);
				\draw[->, dashed] (10) to (11);
				\draw[->, dashed] (11) to node[midway, below] {$G_3$} (12);
				
				\draw[->, dashed] (13) to (16);
				\draw[->, dashed] (16) to node[midway, below] {$G_4$} (15);
				\draw[->, dashed] (15) to (14);	
				
			\end{tikzpicture}
			
			\begin{tikzpicture}[node distance={20mm},main/.style = {draw, circle}] 
				\node[main] (1) {$1$}; 
				\node[main] (2) [right of=1] {$2$}; 
				\node[main] (3) [below of=1] {$3$};
				\node[main] (4) [below of=2] {$4$};
				
				\node[main] (5) [right of=2] {$1$}; 
				\node[main] (6) [right of=5] {$2$}; 
				\node[main] (7) [below of=5] {$3$};
				\node[main] (8) [below of=6] {$4$};
				
				\node[main] (9) [right of=6]{$1$}; 
				\node[main] (10) [right of=9] {$2$}; 
				\node[main] (11) [below of=9] {$3$};
				\node[main] (12) [below of=10] {$4$};
				
				\node[main] (13) [right of=10] {$1$}; 
				\node[main] (14) [right of=13] {$2$}; 
				\node[main] (15) [below of=13] {$3$};
				\node[main] (16) [below of=14] {$4$};
				
				\draw[->, dashed] (1) to (3);
				\draw[->, dashed] (3) to node[midway, below] {$G_5$} (4);
				\draw[->, dashed] (4) to (2);
				
				\draw[->, dashed] (5) to (6);
				\draw[->, dashed, dashed] (6) to (8);
				\draw[->, dashed] (8) to node[midway, below] {$G_6$} (7);
				
				\draw[->, dashed] (9) to (11) node[below right]{$~~~~~G_7$};
				\draw[->, dashed] (11) to (10);
				\draw[->, dashed] (10) to (12);
				
				\draw[->, dashed] (13) to (16);
				\draw[->, dashed] (16) to (14);
				\draw[->, dashed] (14) to (15) node[below right]{$~~~~~G_8$};

			\end{tikzpicture}
			\caption{Feynman's diagrams for Bernoulli case}
		\end{figure}
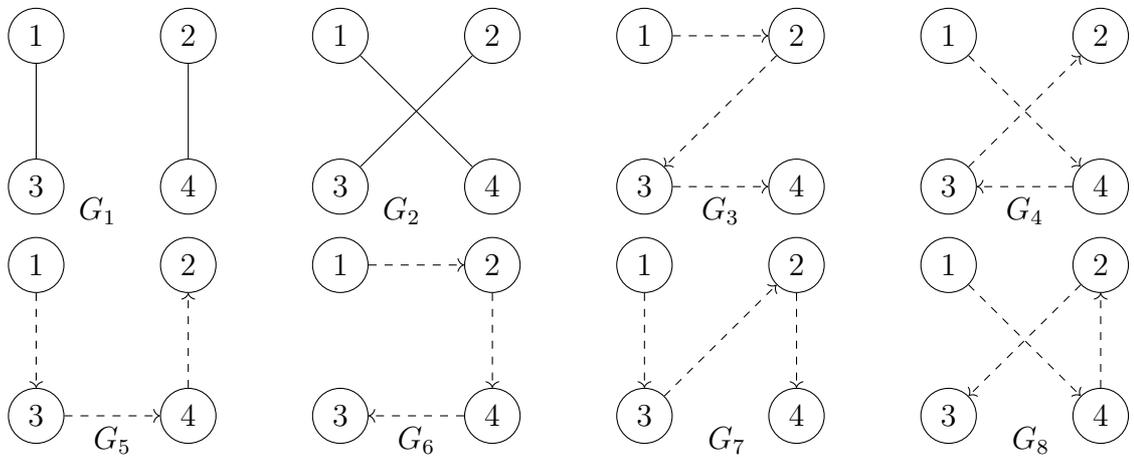
		
		Diagrams with dashed arrows will have zero impact because $ |s| > 2$.
		
		In the first two diagrams, we connect the vertices corresponding to one $f_i$, but keep $n_i$ different edges:
		
		\begin{figure}[h!]
			\centering
			\begin{tikzpicture}[node distance={20mm},main/.style = {draw, circle}] 
				\node[main] (1) {$1$}; 
				\node[main] (2) [right of=1] {$2$}; 
				\node[main] (3) [below of=1] {$3$};
				\node[main] (4) [below of=2] {$4$};
				
				\node[main] (5) [right of=2] {$1$}; 
				\node[main] (6) [right of=5] {$2$}; 
				\node[main] (7) [below of=5] {$3$};
				\node[main] (8) [below of=6] {$4$};
				
				\draw (1) to (3) node[below right]{$~~~G_1$};
				\draw (2) to (4);
				
				\draw (5) to (8);
				\draw (6) to (7) node[below right]{$~~~G_2$};

			\end{tikzpicture}

		$$\Downarrow$$
		
			\begin{tikzpicture}[node distance={20mm},main/.style = {draw, circle}] 
				\node[main] (1) {$f_1$} ; 
				\node[main] (2) [below of=1] {$f_2$};
				\node[main] (3) [right of=1] {$f_1$};
				\node[main] (4) [below of=3] {$f_2$};
				
				\draw (1) to [out=240,in=120,looseness=1]node[near start,  above left ]{$1$}  node[near end,  below left ]{$3$} (2);
				\draw (1) to [out=300,in=60,looseness=1]node[near start,  above right ]{$2$}  node[near end,  below right ]{$4$} (2);
				
				\draw  (3)  to [out=240,in=120,looseness=1]node[near start,  above left ]{$1$}  node[near end,  below left ]{$4$} (4);
				\draw (3)  to [out=300,in=60,looseness=1]node[near start,  above right ]{$2$}  node[near end,  below right ]{$3$}  (4);
				
			\end{tikzpicture}
			\caption{Convergence of diagrams for Bernoulli case to Feynman's diagrams for Gaussian case}
		\end{figure}
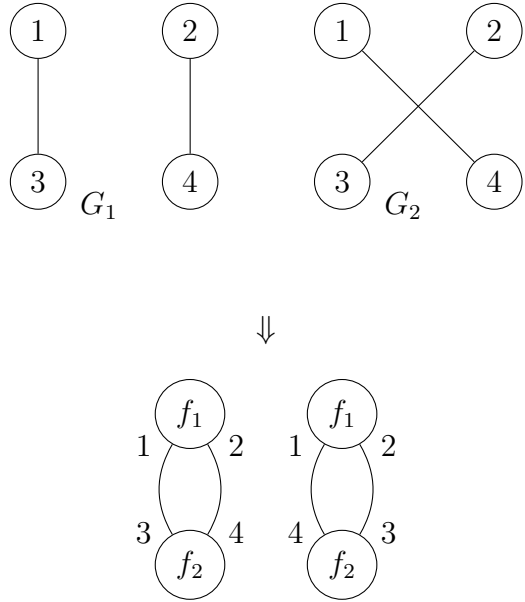 
		
	\end{example}
\newpage
	We obtained exactly the Feynman diagrams for the Gaussian case.
	
\section{Convergence of k-linear forms}

It is known that in the Gaussian case, Wick's powers are the system of orthogonal polynomials and are used to construct a basis in the space of square-integrable functionals from Gaussian white noise. Let's show that Wick's powers in Bernoulli case are not orthogonal.

\begin{example}\label{counterex_orthogonality_bernuolli}

	Using Feynman diagrams for Bernoulli case, calculate 
	$$E\left[:\varphi^3(f): \cdot :\varphi(f):\right].$$
	
	The corresponding diagrams:
	
	
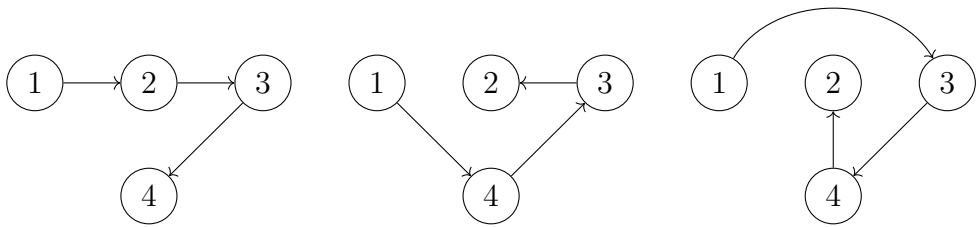
\begin{center}
		\begin{tikzpicture}[node distance={15mm},main/.style = {draw, circle}] 
			\node[main] (1) {$1$}; 
			\node[main] (2) [right of=1] {$2$}; 
			\node[main] (3) [right of=2] {$3$};
			\node[main] (4) [below of=2] {$4$};
			
			\node[main] (5) [right of=3] {$1$}; 
			\node[main] (6) [right of=5] {$2$}; 
			\node[main] (7) [right of=6] {$3$};
			\node[main] (8) [below of=6] {$4$};
			
			\node[main] (9) [right of=7] {$1$}; 
			\node[main] (10) [right of=9] {$2$}; 
			\node[main] (11) [right of=10] {$3$};
			\node[main] (12) [below of=10] {$4$};
			
			\draw[->] (1) to (2);
			\draw[->] (2) to (3);
			\draw[->] (3) to (4);
			
			\draw[->] (5) to (8);
			\draw[->] (8) to (7);
			\draw[->] (7) to (6);
			
			\draw[->] (9) to [out=60,in=120,looseness=1] (11);
			\draw[->] (11) to (12);
			\draw[->] (12) to (10);
			
		\end{tikzpicture}
	\end{center}
	
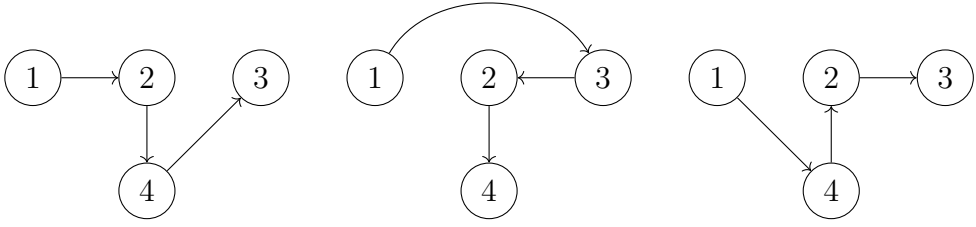
\begin{center}
		\begin{tikzpicture}[node distance={15mm},main/.style = {draw, circle}] 
			\node[main] (1) {$1$}; 
			\node[main] (2) [right of=1] {$2$}; 
			\node[main] (3) [right of=2] {$3$};
			\node[main] (4) [below of=2] {$4$};
			
			\node[main] (5) [right of=3] {$1$}; 
			\node[main] (6) [right of=5] {$2$}; 
			\node[main] (7) [right of=6] {$3$};
			\node[main] (8) [below of=6] {$4$};
			
			\node[main] (9) [right of=7] {$1$}; 
			\node[main] (10) [right of=9] {$2$}; 
			\node[main] (11) [right of=10] {$3$};
			\node[main] (12) [below of=10] {$4$};
			
			\draw[->] (1) to (2);
			\draw[->] (2) to (4);
			\draw[->] (4) to (3);
			
			\draw[->] (5) to [out=60,in=120,looseness=1] (7);
			\draw[->] (7) to (6);
			\draw[->] (6) to (8);
			
			\draw[->] (9) to (12);
			\draw[->] (12) to (10);
			\draw[->] (10) to (11);
		\end{tikzpicture}
	\end{center}
	
	$$I(G_1) = I(G_2) = \sum\limits_{k=1}^n f^4 \left(\frac{k}{n}\right)  \frac{1}{n},$$
	$$I(G_3) = I(G_4) = I(G_5) = I(G_6) = -\sum\limits_{k=1}^n f^4 \left(\frac{k}{n}\right)  \frac{1}{n}.$$
	
	So, we get: 
	$$E\left[:\varphi^3(f)::\varphi(f):\right] = \sum\limits_{\left\{G\right\}} I(G) = -2\sum\limits_{k=1}^n f^4 \left(\frac{k}{n}\right)  \frac{1}{n} \neq 0.$$

\end{example}

	\begin{remark}
	    Wick's powers for Bernoulli case are not orthogonal, in contrast to the Gaussian case.
	\end{remark}

Let $\left\{ \varepsilon_n, n \geq 1 \right\}$ be a sequence of independent Bernoulli random variables. Now let's define the following polynomials 
$$A_k^n \left( \overrightarrow{\varepsilon} \right) = \frac{1}{n^{\frac{k}{2}}} \sum\limits_{i_1 \neq \ldots \neq i_n} f \left( \frac{i_1}{n}, \ldots, \frac{i_k}{n} \right) \varepsilon_{i_1} \ldots \varepsilon_{i_k}, ~1 \leq i_j \leq n. $$

\begin{lemma}
    $\left\{A_k^n \left( \overrightarrow{\varepsilon}\right), 1 \leq k \leq n  \right\}$ is a system of orthogonal polynomials. 
\end{lemma}

\begin{proof}

    $$ \forall k \geq 1:~ \mathbb{E} A_k^n \left( \overrightarrow{\varepsilon} \right) = \mathbb{E} \frac{1}{n^{\frac{k}{2}}} \sum\limits_{i_1 \neq \ldots \neq 1_k} f \left( \frac{i_1}{n}, \ldots, \frac{i_k}{n}  \right) \varepsilon_{i_1} \ldots \varepsilon_{i_k} =$$
    $$ = \frac{1}{n^{\frac{k}{2}}} \sum\limits_{i_1 \neq \ldots \neq 1_k} f \left( \frac{i_1}{n}, \ldots, \frac{i_k}{n}  \right) \mathbb{E}\varepsilon_{i_1} \ldots \mathbb{E} \varepsilon_{i_k} = 0$$
    as a result of independence of $\left\{ \varepsilon_n \right\}.$
    
    $$\forall m \neq k:~\mathbb{E} A_k^n \left( \overrightarrow{\varepsilon} \right) A_m^n \left( \overrightarrow{\varepsilon} \right) = $$
    $$= \mathbb{E} \left[  \frac{1}{n^{\frac{k}{2}}} \sum\limits_{i_1 \neq \ldots \neq i_k} f \left( \frac{i_1}{n}, \ldots, \frac{i_k}{n}  \right) \varepsilon_{i_1} \ldots \varepsilon_{i_k}  \frac{1}{n^{\frac{m}{2}}} \sum\limits_{j_1 \neq \ldots \neq j_m} f \left( \frac{j_1}{n}, \ldots, \frac{j_m}{n}  \right) \varepsilon_{j_1} \ldots \varepsilon_{j_m}\right] = $$
    $$=  \frac{1}{n^{\frac{k+m}{2}}} \mathbb{E} \sum\limits_{i_1 \neq \ldots \neq i_k, j_1 \neq \ldots \neq j_m} f \left( \frac{i_1}{n}, \ldots, \frac{i_k}{n}  \right) f \left( \frac{j_1}{n}, \ldots, \frac{j_m}{n}  \right) \varepsilon_{i_1} \ldots \varepsilon_{i_k} \varepsilon_{j_1} \ldots \varepsilon_{j_m}$$
    
    In this sum, the term will give a non-zero contribution to the mathematical expectation only when the indices $\left\{ i_1, \ldots, i_k, j_1, \ldots, j_m \right\}$ are pairwise equal. That is, the term will have the following form:
    $$f^2 \left( \frac{i_1}{n}, \ldots, \frac{i_k}{n}  \right) \varepsilon_{i_1}^2 \ldots \varepsilon_{i_k}^2.$$
    
    But since $m \neq k$, this is impossible. So
    $$\mathbb{E} A_k^n \left( \overrightarrow{\varepsilon} \right) A_n^n \left( \overrightarrow{\varepsilon} \right) = 0.$$
    
    This shows that $\left\{A_k^n \left( \overrightarrow{\varepsilon}\right), 1 \leq k \leq n  \right\}$ is a system of orthogonal polynomials.
\end{proof}	

\begin{definition}
 
    	Suppose $f \in C\left([0,1]^k\right)$ and $\left\{ \psi_m, m \geq 1 \right\}$ is orthonormal basis in $L_2([0,1])$ where $\forall m: \psi_m \in C([0,1])$ and
	    $$f(\overline{x}) = \sum\limits_{m_1 \ldots m_k =1}^{\infty} c_{m_1,\ldots, m_k} \psi_{m_1}(x_1) \ldots \psi_{m_k}(x_k).$$
	    
	    An infinite-dimensional Hermite polynomial related to $f$ is defined as follows:
	    $$f ( \underbrace{\xi, \ldots, \xi}_k ) = \sum\limits_{m_1 \ldots m_k =1}^{\infty} c_{m_1,\ldots, m_k} (\psi_{m_1}, \xi) * \ldots * (\psi_{m_k}, \xi),$$
	    where $(\psi, \xi)$ is an element of Gaussian white noise, $(\psi_{m_1}, \xi) * \ldots * (\psi_{m_k}, \xi)$ is obtained from $(\psi_{m_1}, \xi) \cdot \ldots \cdot (\psi_{m_k}, \xi)$ in which $(\psi_{m_t}, \xi)^r$ is substituted by \textit{r}-th Hermite polynomial $H_r\left((\psi_{m_t}, \xi)\right).$ \cite{Dor}
\end{definition}

\begin{theorem}
    $\forall k \geq 1:$
    $$A_k^n \left( \overrightarrow{\varepsilon} \right) = \frac{1}{n^{\frac{k}{2}}} \sum\limits_{i_1 \neq \ldots \neq i_n} f \left( \frac{i_1}{n}, \ldots, \frac{i_k}{n} \right) \varepsilon_{i_1} \ldots \varepsilon_{i_k} \Longrightarrow f ( \underbrace{\xi, \ldots, \xi}_k ),~n \to \infty,$$
    where $f ( \underbrace{\xi, \ldots, \xi}_k )$ is Hermite polynomial related to $f$ from definition for $A_k^n \left( \overrightarrow{\varepsilon} \right)$.
\end{theorem}

\begin{proof}

    Let's prove a stronger statement. Namely, let's use the method of mathematical induction to prove the convergence of a finite set of polynomials $A_k^n(\overrightarrow{\varepsilon})$.

    Base case, k=1. For any $f \in C([0,1])$: $$ A_1^n(f)= \frac{1}{\sqrt{n}} \sum\limits_{i=1}^n f\left( \frac{i}{n} \right) \varepsilon_i \Longrightarrow (f, \xi),~n \to \infty.$$
    This statement has already been proved in the theorem~\ref{theorem_convergence_bernuolli_to_gaussian}. Also, from remark~\ref{vector_convergence}, one can say that
    
    $\forall j \geq2,~\forall f_1,\ldots, f_j \in C([0,1]):$
    $$\left(A_1^n(f_1),\ldots, A_1^n(f_j)\right) \Longrightarrow \left( (f_1, \xi), \ldots, (f_j, \xi) \right),~n \to \infty.$$
    
    Induction step. Suppose that the following holds:

    $\forall j\geq2,~ \forall p_m \leq k-1~ \forall f_m \in C([0,1]^{p_m}) , m =\overline{1,j}:$
    $$\left(A_{p_1}^n(f_1),\ldots, A_{p_j}^n(f_j)\right) \Longrightarrow \left( (f_1, \xi), \ldots, (f_j, \xi) \right),~n \to \infty,$$
    where $f ( \underbrace{\xi, \ldots, \xi}_{k-1} )$ is Hermite polynomial related to $f \in C([0,1]^{k-1})$ from the left-hand side of the expression. 
    
    Suppose that $f \in C([0,1]^k)$ can be represented as
    $$f(x_1, \ldots, x_k) = \sum\limits_{m_1, \ldots, m_k}^M c_{m_1 \ldots m_k} \psi_{m_1}(x_1) \cdot \ldots \cdot \psi_{m_k}(x_k), $$
    where $\left\{ \psi_m, m \geq 1 \right\}$ is orthonormal basis in $L_2([0,1])$ where for all $m\geq 1$ $ \psi_m \in C([0,1]).$
    Then
    
    $$\frac{1}{n^{\frac{k}{2}}} \sum\limits_{i_1 \neq \ldots \neq i_k} f \left( \frac{i_1}{n}, \ldots, \frac{i_k}{n}  \right) \varepsilon_{i_1} \ldots \varepsilon_{i_k} = $$
    $$ =\frac{1}{n^{\frac{k}{2}}} \sum\limits_{i_1 \neq \ldots \neq i_k} \sum\limits_{m_1, \ldots, m_k = 1}^M c_{m_1 \ldots m_k} \psi_{m_1}\left( \frac{i_1}{n} \right) \ldots \psi_{m_k} \left( \frac{i_k}{m} \right) \varepsilon_{i_1} \ldots \varepsilon_{i_k} = $$
    
    $$= \frac{1}{n^\frac{k}{2}} \sum\limits_{m_k = 1}^M \sum\limits_{i_k =1}^n \psi_{m_k} \left( \frac{i_k}{n} \right) \varepsilon_k \sum\limits_{m_1 \ldots m_{k-1} =1}^M \sum\limits_{i_1 \neq \ldots \neq i_{k-1}} \psi_{m_1}\left( \frac{i_1}{n} \right) \ldots \psi_{m_{k-1}} \left( \frac{i_{k-1}}{n} \right) \cdot$$
    $$ \cdot c_{m_1 \ldots m_k} \varepsilon_{i_1} \ldots \varepsilon_{i_k} -$$
    
    $$ - \frac{1}{n^\frac{k}{2}} \sum\limits_{j=1}^{k-1} \sum\limits_{m_1, \ldots, m_k = 1}^M \sum\limits_{i_1 \neq \ldots \neq i_{k-1}} c_{m_1 \ldots m_k}  \psi_{m_1}\left( \frac{i_1}{n} \right) \ldots \psi_{m_j}\left( \frac{i_j}{n} \right) \ldots \psi_{m_{k-1}} \left( \frac{i_{k-1}}{n} \right) \cdot$$
    
    $$ \cdot \psi_{m_k}\left( \frac{i_j}{n} \right) \varepsilon_{i_1} \ldots \hat{\varepsilon}_{i_j} \ldots \varepsilon_{i_{k-1}}. $$
    Here, the notation $\hat{\varepsilon}_{i_j}$ means that the factor $\varepsilon_{i_j}$ has been removed from the product. 
    
    Due to the induction assumption, the vector of two summands in the previous expression weakly converges, so the limit of the difference of these two summands will converge to the difference of the limits when $n$ goes to $\infty.$
    Consider the second term:
    \begin{gather}
        \frac{1}{n^\frac{k}{2}} \sum\limits_{j=1}^{k-1} \sum\limits_{m_1, \ldots, m_k = 1}^M \sum\limits_{i_1 \neq \ldots i_{k-1}} c_{m_1 \ldots m_k}  \psi_{m_1}\left( \frac{i_1}{n} \right) \ldots \psi_{m_j}\left( \frac{i_j}{n} \right) \ldots \psi_{m_{k-1}} \left( \frac{i_{k-1}}{n} \right) \cdot \nonumber \\
     \cdot  \psi_{m_k}\left( \frac{i_j}{n} \right) \varepsilon_{i_1} \ldots \hat{\varepsilon_{i_j}} \ldots \varepsilon_{i_{k-1}} =
    \end{gather}
    
    \begin{gather}
        = \sum\limits_{m_j}^M \sum\limits_{i_j}^n \psi_{m_j} \left( \frac{i_j}{n} \right) \psi_{m_k} \left( \frac{i_j}{n}\right) \frac{1}{n} \sum\limits_{m_1 \ldots \hat{m}_j \ldots m_{k-1}} c_{m_1 \ldots m_k} \sum\limits_{i_1 \neq \ldots \neq \i_{k-1}} \frac{1}{n^\frac{k-1}{2}} \psi_{m_1}\left( \frac{i_1}{n} \right) \ldots \nonumber \\
        \hat{\psi}_{m_j}\left( \frac{i_j}{n} \right) \ldots \psi_{m_{k-1}}\left( \frac{i_{k-1}}{n} \right) \varepsilon_{i_1} \ldots \hat{\varepsilon}_{i_j} \ldots \varepsilon_{i_{k-1}} - \\
        - \frac{1}{n^{\frac{k}{2}}} \sum\limits_{l \in \left\{ 1, \ldots, k-1\right\} \setminus \left\{j \right\}} \sum\limits_{m_1 \ldots m_k}^M \sum\limits_{i_1 \neq \ldots \neq i_{k-1}} c_{m_1 \ldots m_k} \psi_{m_1}\left( \frac{i_1}{n} \right) \ldots \psi_{m_j}\left( \frac{i_l}{n} \right) \ldots \nonumber \\
        \psi_{m_l}\left( \frac{i_l}{n} \right) \ldots \psi_{m_{k-1}}\left( \frac{i_{k-1}}{n} \right) \psi_{m_k}\left( \frac{i_l}{n} \right) \varepsilon_{i_1} \ldots \hat{\varepsilon}_{i_j} \ldots \varepsilon_{i_{k-1}}
    \end{gather}
    
    The  expression (9) converges to zero. We can show this in the following way: let's fix $l$ and take out all multipliers with $i_l$, as we did before.
    
     $$\frac{1}{n^{\frac{k}{2}}}  \sum\limits_{m_1 \ldots m_k}^M \sum\limits_{i_1 \neq \ldots \neq i_{k-1}} c_{m_1 \ldots m_k} \psi_{m_1}\left( \frac{i_1}{n} \right) \ldots \psi_{m_j}\left( \frac{i_l}{n} \right) \ldots $$ 
    $$  \psi_{m_l}\left( \frac{i_l}{n} \right) \ldots \psi_{m_{k-1}}\left( \frac{i_{k-1}}{n} \right) \psi_{m_k}\left( \frac{i_l}{n} \right) \varepsilon_{i_1} \ldots \hat{\varepsilon}_{i_j} \ldots \varepsilon_{i_{k-1}} = $$
    $$ = \sum\limits_{m_l, m_j, m_k}^M \sum\limits_{i_l=1}^n \psi_{m_j} \left(\frac{i_l}{n}\right) \psi_{m_l} \left(\frac{i_l}{n}\right) \psi_{m_k} \left(\frac{i_l}{n}\right) \varepsilon_{i_l} \frac{1}{\sqrt{n}} \cdot $$
    $$ \frac{1}{n} \cdot \frac{1}{n^{\frac{k-3}{2}}}\sum_{m_p, p\neq l, p\neq j, p \neq k}^M  \sum\limits_{i_1 \neq \ldots \neq i_{k-1}, \hat{i_l},\hat{i_j}, \hat{i_k} } \prod\limits_{p=1}^{k-1} \psi_{m_p}\left(\frac{i_p}{n}\right) \varepsilon_{i_p} - P(\varepsilon, f),$$
    where $P(\varepsilon, f)$ polynomial that must be subtracted to avoid repeating by indexes.
    
   Expression
    
    $$ \frac{1}{n^{\frac{k-3}{2}}}\sum_{m_p, p\neq l, p\neq j, p \neq k}^M  \sum\limits_{i_1 \neq \ldots \neq i_{k-1}, \hat{i_l},\hat{i_j}, \hat{i_k} } \prod\limits_{p=1}^{k-1} \psi_{m_p}\left(\frac{i_p}{n}\right) \varepsilon_{i_p}$$
    weakly converges. Because of the factor $\frac{1}{n}, $ the whole expression converges to zero.  Repeating a similar procedure for $P(\varepsilon, f)$, we get what we were looking for.
    
    So, the limit of the term (7) equals to limit of the expression (8). Now, using the assumption of the induction step one can conclude that the weak limit of
    $$ \frac{1}{n^\frac{k}{2}} \sum\limits_{j=1}^{k-1} \sum\limits_{m_1, \ldots, m_k = 1}^M \sum\limits_{i_1 \neq \ldots i_{k-1}} c_{m_1 \ldots m_k}  \psi_{m_1}\left( \frac{i_1}{n} \right) \ldots \psi_{m_j}\left( \frac{i_j}{n} \right) \ldots \psi_{m_{k-1}} \left( \frac{i_{k-1}}{n} \right) \cdot$$
    $$ \cdot \psi_{m_k}\left( \frac{i_j}{n} \right) \varepsilon_{i_1} \ldots \hat{\varepsilon_{i_j}} \ldots \varepsilon_{i_{k-1}}  = $$
    coincides with the weak limit of
    $$=  \sum\limits_{m_j}^M \sum\limits_{i_j}^n \psi_{m_j} \left( \frac{i_j}{n} \right) \psi_{m_k} \left( \frac{i_j}{n}\right) \frac{1}{n} \sum\limits_{m_1 \ldots \hat{m}_j \ldots m_{k-1}} c_{m_1 \ldots m_k} \sum\limits_{i_1 \neq \ldots \neq \i_{k-1}} \frac{1}{n^\frac{k-1}{2}} \psi_{m_1}\left( \frac{i_1}{n} \right) \ldots $$
    $$\hat{\psi}_{m_j}\left( \frac{i_j}{n} \right) \ldots \psi_{m_{k-1}}\left( \frac{i_{k-1}}{n} \right) \varepsilon_{i_1} \ldots \hat{\varepsilon}_{i_j} \ldots \varepsilon_{i_{k-1}} $$
    
    Using the induction step, we obtain the following:
    
    $$  \sum\limits_{m_1 \ldots m_k = 1}^M c_{m_1 \ldots m_k} \left( \psi_{m_j}, \psi_{m_k} \right) \left( (\psi_{m_1}, \xi) * \ldots * \widehat{(\psi_{m_j}, \xi)} * \ldots (\psi_{m_k}, \xi) \right), n \to \infty.$$
    
    Therefore, we get:
    
    $$\frac{1}{n^{\frac{k}{2}}} \sum\limits_{i_1 \neq \ldots \neq i_k} f \left( \frac{i_1}{n}, \ldots, \frac{i_k}{n}  \right) \varepsilon_{i_1} \ldots \varepsilon_{i_k} \Longrightarrow$$
    $$\Longrightarrow \sum\limits_{m_1 \ldots m_k=1}^M c_{m_1 \ldots m_k} \Big( (\psi_{m_k}, \xi) \left( (\psi_{m_1}, \xi) * \ldots * (\psi_{m_{k-1}}) \right) -$$
    $$-\sum\limits_{j=1}^{k-1} c_{m_1 \ldots m_k} (\psi_{m_j}, \psi_{m_k}) \left( (\psi_{m_1}, \xi) * \ldots * \widehat{(\psi_{m_j}, \xi)} * \ldots * (\psi_{m_{k-1}}, \xi) \right)  = $$
    $$= \sum\limits_{m_1 \ldots m_k=1}^M c_{m_1 \ldots m_k} (\psi_{m_1}, \xi) * \ldots * (\psi_{m_k}, \xi) \overset{def}{=} f(\underbrace{\xi, \ldots, \xi}_k).$$

    Convergence of a vector
    $$\left(A_{p_1}^n(f_{1}),\ldots, A_{p_j}^n(f_{j}) , A_k^n (f_k)\right)$$
    follows from the assumption of induction and from the fact that the weak limit of the polynomial $A_k^n$ coincides with the weak limit of the sum of polynomials of lower degrees.

    We proved the theorem for the case when the function is represented as a finite linear combination of basis vectors. Now we will prove that this statement is also true for the general case.

    For any $f \in C([0,1]^k):$
    $$f(x_1, \ldots, x_k) = \sum\limits_{m_1, \ldots, m_k}^\infty c_{m_1 \ldots m_k} \psi_{m_1}(x_1) \ldots \psi_{m_k}(x_k)$$
    consider the following function:
    $$f^M (x_1, \ldots, x_k) = \sum\limits_{m_1, \ldots, m_k}^M c_{m_1 \ldots m_k} \psi_{m_1}(x_1) \ldots \psi_{m_k}(x_k) .$$
    It has already been shown that
    \begin{gather}
        A_n^k(f^M) \Longrightarrow f^M (\xi, \ldots, \xi), n \to \infty.
    \end{gather}
    Also, from the construction of infinite dimensional Hermite polynomials, we know that
    \begin{gather}
        f^M (\xi, \ldots, \xi) \Longrightarrow f(\xi, \ldots, \xi), M \to \infty.
    \end{gather}

    \begin{gather}
         \overline{\lim\limits_{n \to \infty}} \mathbb{E} \left( A_n^k (f^M - f) \right)^2 = \overline{\lim\limits_{n \to \infty}} \frac{1}{n^k}\sum\limits_{i_1 \neq \ldots \neq i_k}^n \left( (f^M - f)\left( \frac{i_1}{n}, \ldots, \frac{i_k}{n} \right)\right)^2 \leq \\ \nonumber
        \leq \overline{\lim\limits_{n \to \infty}} \frac{1}{n^k}\sum\limits_{i_1, \ldots, i_k = 1}^n \left( (f^M - f)\left( \frac{i_1}{n}, \ldots, \frac{i_k}{n} \right)\right)^2 = ||f^M - f||^2 \longrightarrow 0, M \to \infty.
    \end{gather}
   Let's show that for any bounded function $\Phi$ which satisfies the Lipschitz condition 
    $$\mathbb{E} \Phi \left( A_n^k (f) \right) \longrightarrow \mathbb{E} \Phi \left( f(\xi, \ldots, \xi) \right), n \to \infty.$$
    Because of (11)  for any $\delta >0 $ there exists $M_0$ such that
    
    $\forall M \geq M_0:$
    $$|\mathbb{E} \Phi \left( f(\xi, \ldots, \xi) \right) - \mathbb{E} \Phi \left( f^M(\xi, \ldots, \xi) \right)| < \delta$$
    For such $M \geq M_0$ choose $n_0$ such that

    $\forall n \geq n_0:$

   $$|\mathbb{E} \Phi \left( A_n^k(f) \right) - \mathbb{E} \Phi \left( A_n^k(f^M) \right) | <\delta$$ 
   and
    $$|\mathbb{E} \Phi \left( A_n^k(f^M) \right) - \mathbb{E} \Phi \left( f^M(\xi, \ldots, \xi) \right) | <\delta.$$
    We can do this because of (10) and (12).
    Then, for any $\delta >0$ and sufficiently large $M$ 
    
    $\exists n_0~ \forall n \geq n_0:$
    $$|\mathbb{E} \Phi \left( A_n^k (f) \right) - \mathbb{E} \Phi \left( f(\xi, \ldots, \xi) \right)| \leq$$
    $$\leq |\mathbb{E} \Phi \left( f(\xi, \ldots, \xi) \right) - \mathbb{E} \Phi \left( f^M(\xi, \ldots, \xi) \right)| + $$
    $$ + |\mathbb{E} \Phi \left( A_n^k(f) \right) - \mathbb{E} \Phi \left( A_n^k(f^M) \right) | +$$
    $$ + |\mathbb{E} \Phi \left( A_n^k(f^M) \right) - \mathbb{E} \Phi \left( f^M(\xi, \ldots, \xi) \right) |  \leq 3\delta.$$
    So, we proved that
    $$A_n^k (f) \Longrightarrow f(\xi, \ldots, \xi), n \to \infty.$$

\end{proof}

\end{document}